\documentclass[11pt,reqno]{amsart}
\usepackage{amscd,amsmath,amsopn,amssymb,amsthm,multicol}
\usepackage[breaklinks=true,colorlinks=true,linkcolor=blue,citecolor=blue,urlcolor=blue]{hyperref}

\DeclareMathOperator{\ad}{ad}

\DeclareMathOperator{\Ad}{Ad}

\DeclareMathOperator{\Isom}{Isom}
\DeclareMathOperator{\trace}{trace}
\DeclareMathOperator{\rk}{rk}

\renewenvironment{proof}[1][Proof]{\textbf{#1.} }
{\ \rule{0.5em}{0.5em}}
\newtheorem{theorem}{Theorem}
\newtheorem{proposition}{Proposition}
\newtheorem{lemma}{Lemma}
\newtheorem{corollary}{Corollary}
\newtheorem{question}{Question}

\theoremstyle{definition}

\newtheorem{remark}{Remark}
\newtheorem{example}{Example}

\begin{document}

\title
[On homogeneous geodesics and weakly symmetric spaces \dots] {On homogeneous geodesics and \\ weakly symmetric spaces}

\author{Valeri\u{\i} Berestovski\u{\i}}
\address{V.\,N. Berestovski\u{\i} \newline
Sobolev Institute of Mathematics of the Siberian Branch \newline
of the Russian Academy of Sciences, Novosibirsk, Acad. Koptyug ave. 4, \newline
630090, Russia}
\address{Novosibirsk State University, \newline Mechanics-Mathematical department
\newline 2 Pirogov str., Novosibirsk, 630090, Russia}
\email{vberestov@inbox.ru}

\author{Yuri\u{\i}~Nikonorov}
\address{Yu.\,G. Nikonorov \newline
Southern Mathematical Institute of the Vladikavkaz Scientific Centre \newline
of the Russian Academy of Sciences, Vladikavkaz, Markus st. 22, \newline
362027, Russia}
\email{nikonorov2006@mail.ru}

\begin{abstract}

In this paper, we establish a sufficient condition for a geodesic in a Riemannian manifold
to be homogeneous, i.~e. an orbit of an $1$-parameter isometry group. As an application of
this result, we provide a new proof of the fact that every weakly symmetric space is geodesic
orbit manifold, i.~e. all its geodesics are homogeneous. We also study general properties of homogeneous geodesics,
in particular, the structure of the closure of a given homogeneous geodesic. We present
several examples where this closure is a torus of dimension $\geq 2$ which is (respectively, is not) totally
geodesic in the ambient manifold. Finally, we discuss homogeneous
geodesics in Lie groups supplied with left-invariant Riemannian metrics.

\vspace{2mm} \noindent 2010 Mathematical Subject Classification:
53C20, 53C25, 53C35.

\vspace{2mm} \noindent Key words and phrases: geodesic orbit Riemannian space, homogeneous
Riemannian manifold, homogeneous space, quadratic mapping, totally geodesic torus,
weakly symmetric space.
\end{abstract}

\maketitle

\section{Introduction, notation and useful facts}

Let $(M,g)$ be a Riemannian manifold and let $\gamma: \mathbb{R} \rightarrow M$ be a
geodesic in  $(M,g)$. The geodesic $\gamma$ is called {\it homogeneous} if
$\gamma(\mathbb{R})$ is an orbit of an $1$-parameter subgroup of $\Isom(M,g)$, the full
isometry group of $(M,g)$. A Riemannian manifold $(M,g)$ is called a  manifold with
homogeneous geodesics or a geodesic orbit manifold if any
geodesic $\gamma$ of $M$ is homogeneous.

These definitions are naturally generalized to the case when all isometries are taken from
a given Lie subgroup $G\subset \Isom(M,g)$, that acts transitively
on $M$. In this case we get the notions of $G$-homogeneous geodesics and $G$-homogeneous
geodesic orbit spaces. This terminology was introduced in \cite{KV} by O.~Kowalski and
L.~Vanhecke, who initiated a systematic study of such spaces.
We refer to \cite{KV}, \cite{Arv2017}, and \cite{Nikonorov2017} for expositions on
general properties of geodesic orbit Riemannian spaces and some historical survey about
this topic.

Let $(M=G/H, g)$ be a homogeneous Riemannian manifold, where $G$ is the identity
component of $\Isom(M,g)$ and $H$ is the isotropy subgroup at a point $o\in M$.
Since $H$ is compact, there is an $\Ad(H)$-invariant decomposition
\begin{equation}\label{reductivedecomposition}
\mathfrak{g}=\mathfrak{h}\oplus \mathfrak{m},
\end{equation}
where $\mathfrak{g}={\rm Lie }(G)$ and $\mathfrak{h}={\rm Lie}(H)$.
The Riemannian metric $g$ is $G$-invariant and is determined
with an $\Ad(H)$-invariant inner product $g = (\cdot,\cdot)$ on
the space $\mathfrak{m}$ which is identified with the tangent
space $M_0$ at the initial point $o = eH$.
By $[\cdot, \cdot]$ we denote the Lie bracket in $\mathfrak{g}$, and by
$[\cdot, \cdot]_{\mathfrak{m}}$ its $\mathfrak{m}$-component
according to (\ref{reductivedecomposition}).
We recall (in the above terms) a well-known criterion of geodesic orbit spaces.

\begin{lemma}[\cite{KV}]\label{GO-criterion}
A homogeneous Riemannian manifold   $(M=G/H,g)$ with the reductive
decomposition  {\rm (\ref{reductivedecomposition})} is a geodesic orbit space if and
only if  for any $X \in \mathfrak{m}$ there is $Z \in \mathfrak{h}$ such that
$([X+Z,Y]_{\mathfrak{m}},X) =0$ for all $Y\in \mathfrak{m}$.
\end{lemma}

For a given $X\in\mathfrak{m}$, the condition $([X+Z,Y]_{\mathfrak{m}},X) =0$ for
all $Y\in \mathfrak{m}$ means that the orbit of  $\exp\bigl((X+Z)t\bigr) \subset G$,
$t \in \mathbb{R}$, through the point $o=eH$ is a geodesic in $(M=G/H,g)$.
Note also that all orbits of an $1$-parameter isometry group, generated with a
Killing vector fields of constant length on a given Riemannian manifold,
are geodesics, see e.~g. \cite{BerNik7}.

An important class of geodesic orbit manifolds consists of weakly symmetric spaces,
introduced by A.~Selberg \cite{S}. A homogeneous Riemannian manifold  $(M = G/H, g)$ is a
{\it  weakly symmetric space} if any two points $p,q \in M$ can be interchanged with
an isometry $a \in G$ (this is a definition equivalent to the original one).
Note that a Riemannian manifold $(M,g)$ is a weakly symmetric space if and only if
it is homogeneous and for some (hence, every) point $x \in M$ and any reductive
decomposition (\ref{reductivedecomposition}) the following property holds:
for every $U\in \mathfrak{m}$ there is $s \in H$ such that $\Ad(s) (U)=-U$~\cite{Zi96}.
Note that every $G$-invariant Riemannian metric on $G/H$ (with the above property) makes
it a weakly symmetric space.

Weakly symmetric spaces $M= G/H$ have many interesting properties and  are closely related
with spherical spaces, commutative spaces, Gelfand pairs etc. (see papers
\cite{AV, BV1996, Yak, Zi96} and book \cite{W1} by J.A.~Wolf). The classification of
weakly symmetric reductive homogeneous  Riemannian spaces was  given by O.S.~Yakimova
\cite{Yak} on the base of the paper \cite{AV} (see also  \cite{W1}).

Let us recall that {\it weakly symmetric Riemannian manifolds are geodesic orbit}
by a result of J.~Berndt, O.~Kowalski, and L.~Vanhecke \cite{BKV}. The main motivation of
this paper was to reprove this result by alternative simple methods.

It should be noted that the full isometry group $\Isom (M,g)$ of a given Riemannian manifold
$(M,g)$ is a Lie group and the isotropy subgroup at any point $x\in M$ is compact
by the Myers--Steenrod theorem. Note also that by the Cartan theorem, a closed (abstract) subgroup
of a Lie group is a Lie subgroup, hence, Lie group itself,
but this is not true in general for non-closed subgroup, see details e.~g. in \cite{HilNeeb}.

All manifolds in this paper are supposed to be connected. For a smooth manifold $M$ and
$x\in M$, $M_x$ denotes the tangent space to $M$ at the point $x$.
For a smooth manifolds mapping $\psi : M \rightarrow N$, we denote by $T\psi$ its differential.

The structure of this paper is as follows.

In Section \ref{mr}, we prove in Theorem \ref{isomhomgeod} that a geodesic $\gamma$ in
a given smooth Riemannian manifold $(M,g)$ is homogeneous if the set (group) $G_{\gamma}$ of all
isometries in $(M,g)$, preserving $\gamma$ and its orientation, acts transitively on
$\gamma$. Recall that $\gamma$ is homogeneous if it is an orbit of a point $x\in \gamma$
under an 1-parameter Lie subgroup $\psi(t)$, $t\in \mathbb{R}$, of the full isometry
Lie group $\Isom(M,g)$ of $(M,g)$.
At first we give two alternative proofs of Proposition \ref{isomhomgeodn} which states
that $G_{\gamma}$ is a closed (hence, Lie) subgroup of $\Isom(M,g)$. The first proof of
this proposition uses some results from the theory of topological groups while the
second one applies the so-called {\it development of the geodesic} $\gamma$ in
$p^{-1}(\gamma)$ where $p:TM\rightarrow M$ is the canonical projection. After this it
is quite easy to prove Theorem \ref{isomhomgeod} and then Theorem~\ref{cor.weaksymgo}
stating that every weakly symmetric Riemannian space is geodesic orbit. At the
end of this section, we discuss briefly some known results on geodesics invariant
under distinguished isometry of $(M,g)$.

In Section \ref{cl}, the closure of a homogeneous geodesic $\gamma$ in $(M,g)$ and the
corresponding 1-parameter group $\psi$ in $\Isom(M,g)$ are investigated. It is proved
in Proposition \ref{prop.closopsub} that the closure of $\psi(\mathbb{R})$ in $\Isom(M,g)$
coincides with $\psi(R)$ or is isomorphic to a compact commutative Lie group (torus $T^k$)
for $k\geq 2$. By Theorem \ref{nontrisotr}, the same statement is true for~$\gamma$ (in
general case, it is possible that the dimension of the closure $T^k$ of $\psi(\mathbb{R})$
in $\Isom(M,g)$ is greater than the dimension of the closure $T^l$ of $\gamma(\mathbb{R})$
in $(M,g)$); if $l\geq 2$ then sectional curvatures of all 2-planes tangent to $T^l$,
calculated in $(M,g)$, are nonnegative. Then we present some examples of $T^{l}$, $l\geq 2$,
that are (respectively, are not) totally geodesic in $(M,g)$. At the end of this section
are given some references to papers containing interesting results on geodesics.

In Section \ref{ehg}, we present some examples of homogeneous geodesics $\gamma$ on Lie groups
$G$ with left invariant Riemannian metric $g$, among them such that the corresponding
torus $T^l$ for non-closed subset $\gamma(\mathbb{R})\subset (G,g)$ is (respectively,
is not) totally geodesic in $(G,g)$.

In Section \ref{sqm}, we study properties of a special quadratic mapping closely related to homogeneous geodesics in $(G,g)$.

\section{Main results}
\label{mr}

\begin{theorem}\label{isomhomgeod}
Let $(M,g)$ be a Riemannian manifold, and let $\gamma: \mathbb{R} \rightarrow M$ be a
geodesic parameterized with arc length. Suppose that for any $s\in \mathbb{R}$ there is an
isometry $\eta(s)\in \Isom(M,g)$, such that $\eta(s)(\gamma(t)) =\gamma(t+s)$ for all
$t \in \mathbb{R}$. Then the geodesic $\gamma$ is homogeneous, i.~e. an orbit of an
1-parameter isometry group.
\end{theorem}

\begin{proof}
Let us consider
\begin{eqnarray}\label{eq.onepariso1}
G_{\gamma}(s)&:=&\{a\in \Isom(M,g)\, |\,  a(\gamma(t)) =\gamma(t+s)\, \forall t \in \mathbb{R}\}, \quad s\in \mathbb{R};\\
\label{eq.onepariso2}
G_{\gamma}&:=&\cup_{s\in \mathbb{R}}\, G_{\gamma}(s).
\end{eqnarray}
We know that $G_{\gamma}(s) \neq \emptyset$ for all $s\in \mathbb{R}$ and
$G_{\gamma}=\cup_{s\in \mathbb{R}}\, G_{\gamma}(s)$.
Clear that $G_{\gamma}(0)$ is compact, since it is the intersection of the isotropy
subgroups at all points of $\gamma(\mathbb{R})$ with respect to $\Isom(M,g)$.
It is obvious that $G_{\gamma}$ is a subgroup in $\Isom(M,g)$.

The most crucial step in this proof is to prove that $G_{\gamma}$ is a Lie group.
We prove this fact separately in Proposition \ref{isomhomgeodn} below.

\smallskip

Let $G_{\gamma}^0$ be the identity component of $G_{\gamma}$. Since $G_{\gamma}$
is a Lie group, $G_{\gamma}^0$ is also a Lie group, hence, a Lie (possibly, virtual) subgroup
in $\Isom(M,g)$. Therefore, for any $U\in \mathfrak{g}$, where $\mathfrak{g}$ is the Lie
algebra of the group $G_{\gamma}^0$, we get that $\exp(tU) \subset G_{\gamma}^0$.
Clearly, there are $U \in \mathfrak{g}$ and $t\in \mathbb{R}$ such that
$\exp(tU) \not\subset G_{\gamma}(0)$.
Hence, $\gamma(\mathbb{R})=\exp(tU)(\gamma(0))$, $t\in \mathbb{R}$, as required.
\end{proof}

\begin{proposition}\label{isomhomgeodn}
The group $G_{\gamma}$ defined with (\ref{eq.onepariso2}) is a Lie group.
\end{proposition}

\begin{proof}
If $\gamma$ is non-injective then the properties of $G_{\gamma}$ (see (\ref{eq.onepariso1})
and (\ref{eq.onepariso2})) imply that $\gamma$ is
periodic i.~e. there is the smallest $s>0$  such that  $\gamma(t+s)=\gamma (t)$ for all
$t \in \mathbb{R}$ (recall that the motion of a point $\gamma(t)$ along $\gamma$ under
the action of $G_{\gamma}$ depends essentially only on its position in $M$,
but not on the value of $t$).
Then $G_{\gamma}$ is compact, hence, a Lie subgroup in $\Isom(M,g)$.
\smallskip

From this point we suppose that $\gamma$ is injective. In this case we give two proofs of
the fact that $G_{\gamma}$ is a Lie group.

{\it The first proof.}
Let us supply the group $G_{\gamma}$ with {\it the natural topology}. Define the natural projection
$$
\eta: G_{\gamma} \mapsto \mathbb{R},
$$
as follows: $\eta(a)=s$ if and only if $a(\gamma(t))=\gamma (t+s)$ for all
$t \in \mathbb{R}$. Clearly, such $s$ is unique.
Let $\Pi_s$ the parallel transport  of $M_{\gamma(0)}$ to $M_{\gamma(s)}$ along
$\gamma$ with respect to the Levi-Civita connection of the Riemannian manifold $(M,g)$.
Consider any $a\in G_{\gamma}$ and put $s=\eta(a)$. Let us define the map
$a\in G_{\gamma} \mapsto \varphi(a)\in O(M_{\gamma(0)})$ as follows:
$$
\varphi(a)=  \Pi_s^{-1} \circ (T\, a)_{\gamma(0)}.
$$
Let us supply $G_{\gamma}$ with the topology  induced with the product topology on the space
$\mathbb{R} \times O(M_{\gamma(0)})$ under the mapping
$a\in G_{\gamma} \mapsto (\eta(a),\varphi(a))\in \mathbb{R} \times O(M_{\gamma(0)})$,
that is obviously injective. This topology makes $G_{\gamma}$ a locally compact topological group.
\smallskip

It is clear that $\eta(a_1\cdot a_2)=\eta(a_1)+\eta(a_2)$ for $a_1, a_2 \in G_{\gamma}$.
Since $\eta$ is an open surjective homomorphism of topological groups,
then $\operatorname{Ker} (\eta)$ is a normal subgroup in $G_{\gamma}$. If fact,
$\operatorname{Ker}(\eta)= G_{\gamma}(0)$ is compact.
The following important result  is known in the theory of topological groups:
If $G$ is a topological group and $H$ is a closed invariant subgroup of $G$
such that $H$ and $G/H$ are Lie groups, then $G$ is a Lie group
(see Theorem 1 in \cite{Gleason1949}, Theorem 7 in \cite{Iwasawa1949}, or
pp.~153--154 in \cite{MontZip1955}).
Since $\mathbb{R}=G_{\gamma}/G_{\gamma}(0)$, where $\mathbb{R}$ and $G_{\gamma}(0)$ are Lie groups,
then  $G_{\gamma}$ is a Lie group by the above result.
\smallskip

{\it The second proof.}
We will use the development of the geodesic $\gamma$ in $p^{-1}(\gamma)$ where
$p:TM\rightarrow M$ is the canonical projection.

Let $\Pi_{t',\,t}$ be the parallel transport of $M_{\gamma(t)}$ to $M_{\gamma(t')}$ along $\gamma$
with respect to the Levi-Civita connection of the Riemannian manifold $(M,g)$.
Obviously, $\Pi_{t',\, t}\gamma\,'(t)=\gamma\,'(t')$. Similar statement is valid for all $T\psi$,
$\psi\in G_{\gamma}$.

Let $W=\cup_{t\in \mathbb{R}}W_{\gamma(t)}$,
$W_{\gamma(t)}=\{w\in M_{\gamma(t)}: w\perp \gamma\,'(t)\}$.
Choose any orthonormal basis $e_2,\dots, e_n$ in $W_{\gamma(0)}$ and define a basis
$e_2(t),\dots, e_n(t)$ of $W_{\gamma(t)}$ by equalities $e_j(t)=\Pi_{t,0}(e_j)$, $j=2,\dots, n$.
Then $W$ is a smooth vector bundle over $\gamma$ (the topology on $\gamma$ is defined
with the parameter $t$) with the restriction of $p$ to $W$, and $\Pi_{t',t}$, $T\psi$, where
$\psi\in G_{\gamma}$, are smooth linear isomorphisms on $W$. In addition, for any
$t,t',s\in \mathbb{R}$ and $\psi\in G_{\gamma}(s)$,
$T\psi_{\gamma(t')}\circ\Pi_{t',\,t}=\Pi_{t'+s,\,t+s}\circ T\psi_{\gamma(t)}$.
Since the matrix of any parallel transport $\Pi_{r',\,r}$ with respect to bases
$e_{2}(r),\dots,e_n(r)$ and $e_{2}(r'),\dots,e_n(r')$ is always the unit matrix, this
implies that the matrix of
$T\psi_{\gamma(t)}$ with respect to bases $e_{2}(t),\dots,e_n(t)$ and
$e_{2}(t+s),\dots,e_n(t+s)$ coincides with the matrix of
$T\psi_{\gamma(t')}$ with respect to bases $e_{2}(t'),\dots,e_n(t')$
and $e_{2}(t'+s),\dots,e_n(t'+s)$. Denote this matrix, which is independent on $t$, by $(\psi)$.

The above argument shows that any element $\psi\in G_{\gamma}(s)$ and the action of $T\psi$ on
$W$ are uniquely defined with the pair $(s,(\psi))$. In this notation,
the product in the group $T\psi$, $\psi\in G_{\gamma}$, is written as
\begin{equation}
\label{actmatr}
(s,(\psi))(s\,',(\psi'))=(s+s\,',(\psi)(\psi')).
\end{equation}
Thus, if $A(s)=\{(\psi): \psi\in G_{\gamma}(s)\}$ then
\begin{equation}
\label{as}
A(s+s')=A(s)A(s')=A(s')A(s), \quad A(-s)=A(s)^{-1}, \quad A(s)=A_sA(0)
\end{equation}
for $s, s'\in \mathbb{R}$, $A_s\in A(s)$.
Then $A(s)$ is compact for any $s\in \mathbb{R}$, since $A(0)$ is compact,
and $A(-s)=A(0)^{-1}A^{-1}=A(0)A^{-1}$. The last equalities implies that $A(s)=AA(0)=A(0)A$ for
any $A\in A(s)$ and $s\in \mathbb{R}$ and $A(s+t)=A_sA(t)$ for any $A_s\in A(s)$ and
$s,t\in \mathbb{R}$. In particular, $A(0)$ is a compact normal subgroup
of the group $\mathcal{A}=\{A(s), s\in \mathbb{R}\}\subset O(n-1)$.

Let $d$ be the intrinsic metric on $(M,g)$, $\delta$ any bi-invariant metric on the
group $O(n-1)$, whose restriction to $SO(n-1)$ coincides with the intrinsic metric defined with a
bi-invariant Riemannian metric on $SO(n-1)$ and $\delta_H$ the corresponding Hausdorff metric
on the family of compact subsets in $O(n-1)$. Note that
$\delta_H(A(s),A(0))=\delta(A_s,A(0))=\delta(A_0,A(s))$ for any $A_s\in A(s)$ and $A_0\in A(0)$
because of the last equality in (\ref{as}) and the right invariance of the metric $\delta$.
We state that $\delta_H(A(s),A(0))\rightarrow 0$
when $s\neq 0$ and $s\rightarrow 0$. Otherwise, there is a sequence
$\psi_{s_n}\in G_{\gamma}(s_n)$ such that $s_n\neq 0$,
$s_n\rightarrow 0$, $\delta((\psi_{s_n}),A(0))>\varepsilon$ for some
$\varepsilon >0$ and all $n\in \mathbb{N}$, and $d(\psi_{s_n}(x),\psi(x))\rightarrow 0$ for some
$\psi\in \Isom(M,g)$ uniformly for $x\in B(\gamma(0),r)$, where $0< r < \infty$. Then
$\psi \in G_{\gamma}(0)$ but $(\psi)\notin A(0)$, a contradiction. Therefore, for any
$s_0\in \mathbb{R}$ and $A_{s_0}\in A(s_0)$,
\begin{equation}
\label{hausd}
\delta_H\bigl(A(s_0+s),A(s_0)\bigr)=\delta_H\bigl(A_{s_0}A(s),A_{s_0}A(0)\bigr)=\delta_H\bigl(A(s),A(0)\bigr)\rightarrow 0
\end{equation}
if $s\neq 0$ and $s\rightarrow 0$.

The orthonormal bases $e_2(t),\dots, e_n(t)$ in $W_{\gamma(t)}$ permit to consider
$W=\cup_{s}W_{\gamma(s)}$ as a direct product $\mathbb{R}\times W_{\gamma(0)}$ and
supply the last manifold by the direct product of the standard Riemannian metrics
on its factors. Then $W$ is isometric to $n$-dimensional Euclidean space. If
$\psi\in G_{\gamma}(s)$ we define the Euclidean motion in $W$ by the formula
\begin{equation}
\label{actionG}
\psi(t,w)=\bigl(s+t,(\psi)w\bigr),
\end{equation}
where the vector $w\in W_{\gamma(0)}$ is considered as a vector-column with components in the base
$e_2(0),\dots, e_n(0)$.

In consequence of (\ref{hausd}) and compactness of sets $A(s)\subset O(n-1)$, the correspondence
$\psi\in G_{\gamma}(s)\rightarrow (s,(\psi))$ and formulae (\ref{actmatr}), (\ref{actionG}) give
the exact representation of $G_{\gamma}$ as a closed, hence a Lie, subgroup of $\Isom(W)$ for
$n$-dimensional Euclidean space $W$.
\end{proof}
\smallskip

Theorem \ref{isomhomgeod} implies a new proof of the following important result, that was obtained
in~\cite{BKV} using other methods.

\begin{theorem}[J.~Berndt--O.~Kowalski--L.~Vanhecke, \cite{BKV}]\label{cor.weaksymgo}
Every weakly symmetric Riemannian space $(M,g)$ is geodesic orbit.
\end{theorem}

\begin{proof}
Let us fix a geodesic $\gamma: \mathbb{R} \rightarrow M$ in a weakly symmetric Riemannian
manifold $(M,g)$. For any $p \in \gamma(\mathbb{R})$, there is a non-trivial isometry
$\eta(p)\in \Isom( M,g)$ that is a nontrivial involution on $\gamma(\mathbb{R})$ fixing the point
$p \in \gamma(\mathbb{R})$ (see e.~g.~\cite{Zi96}).

For a given $s\in \mathbb{R}$, the isometry $\psi(s):=\eta(\gamma(s/2))\circ \eta(\gamma(0))$
preserves $\gamma(\mathbb{R})$ and its orientation, and moves
the point $\gamma(0)$ to $\gamma(s)$. Therefore,
the geodesic $\gamma$ is homogeneous by Theorem \ref{isomhomgeod}.
\end{proof}

\bigskip

Let $(M,g)$ be a Riemannian manifold, and let $\gamma: \mathbb{R} \rightarrow M$ be a
geodesic parameterized with arc length.
The geodesic $\gamma$ is called {\it invariant under the isometry} $a\in \Isom(M,g)$, if
there is $\tau \in \mathbb{R}$ such that $a(\gamma(t)) =\gamma(t+\tau)$ for all $t \in \mathbb{R}$.

If $\gamma$ is a homogeneous geodesic, then, according to Proposition \ref{isomhomgeodn},
the isometry group $a\in \Isom(M,g)$ such that $\gamma$ is invariant under $a$, is a Lie group
$G_{\gamma}$ defined with (\ref{eq.onepariso2}). Moreover, by the proof of
Proposition \ref{isomhomgeodn}, $G_{\gamma}= G_{\gamma}(0) \times \mathbb{R}$, where
$$G_{\gamma}(0)=\{a\in \Isom(M,g)\,|\, a(\gamma(t)) =\gamma(t)\, \forall t \in \mathbb{R}\}.$$

Geodesics invariant under a distinguished isometry are studied in various papers,
see e.~g. \cite{Grove1974, Bangert2016} and references therein.
In particular, K.~Grove proved the following result.

\begin{theorem}[\cite{Grove1974}]
If $M$ is compact and the isometry  $A\in \Isom(M,g)$ has a non-closed invariant geodesic
then there are uncountably many $A$-invariant geodesics on $M$.
\end{theorem}

It should be noted also the following recent result by V.~Bangert.

\begin{theorem}[\cite{Bangert2016}]
Let $\gamma$ be a non-closed and bounded geodesic in a complete Riemannian
manifold $(M,g)$ and assume that $\gamma$ is invariant under an isometry $A$ of
$(M,g)$, but is not contained in the set of fixed points of $A$. Then for some $k\geq 2$,
the geodesic line flow $\gamma'$ corresponding to $\gamma$ is dense in
a $k$-dimensional torus $T^k$ embedded in $TM$ and, in particular, every geodesic with
initial vector in $T^k$ is $A$-invariant.
\end{theorem}

\section{On the closure of a homogeneous geodesic}
\label{cl}

Now, we are going to discuss important properties of an arbitrary homogeneous geodesic
$\gamma$ on a given Riemannian manifold $(M,g)$.
The main object of our interest is the closure of a given homogeneous geodesic.
The following result is well known (see e.~g. \cite{Goto1971}), but we give a short
proof for the reader's convenience.

\begin{proposition}\label{prop.closopsub}
Let $\psi(s)=\exp(sU)$, $s\in \mathbb{R}$, be an 1-parameter group in a given Lie group~$G$, where $U$ is from $\mathfrak{g}$, the Lie algebra of $G$.
Then $\psi(\mathbb{R})$ is a connected abelian subgroup of $G$ and there are three
possibilities:

1) $\psi(\mathbb{R})$ is a closed subgroup of $G$, diffeomorphic to $\mathbb{R}$
($\psi$ is a diffeomorphism);

2) $\psi(\mathbb{R})$ is a closed subgroup of $G$, diffeomorphic to the circle
$S^1$ ($\psi$ is a covering map);

3) $\psi(\mathbb{R})$ is not closed subgroup of $G$, and its closure is a torus
$T^k$ of dimension $k\geq 2$.

\end{proposition}

\begin{proof}
Clear that $K$, the closure of $\psi(\mathbb{R})$ in $G$, is a connected abelian Lie
group. If $K=\psi(\mathbb{R})$, then we get either the first
or the second possibility.
Suppose that $K\neq \psi(\mathbb{R})$. If $K$ is not a torus, then
$K= \mathbb{R}\times K_1$ for some abelian connected group $K_1$.
If $\pi:\mathbb{R}\times K_1 \rightarrow \mathbb{R}$ is the projection to the first
factor, then $\pi \circ \psi: \mathbb{R} \rightarrow  \mathbb{R}$
is a Lie group isomorphism. Now, consider any point $(a,b)\in \mathbb{R}\times K_1$.
There is a sequence
$(a_n,b_n)\in \psi(\mathbb{R})\subset \mathbb{R}\times K_1$ such that
$\lim\limits_{n\to\infty}(a_n,b_n)=(a,b)$.
It is clear that $(a_n,b_n)\in \psi([-M,M])$ for some positive $M\in \mathbb{R}$.
Indeed,  $a_n \to a$ as $n \to \infty$ and $\pi \circ \psi$ is a Lie group automorphism
of $\mathbb{R}$, hence, the set
$(\pi \circ \psi)^{-1}(a_n)=\psi^{-1}\bigl((a_n,b_n)\bigr)$, $n\in \mathbb{N}$, is
bounded. Since $\psi([-M,M])$ is compact, then $(a,b)\in \psi(\mathbb{R})$ and
$K=\psi(\mathbb{R})$ that impossible. Hence, $K$ is a torus $T^k$ of dimension $k\geq 1$.
Obviously, $K\neq \psi(\mathbb{R})$ implies $k\geq 2$.
\end{proof}
\smallskip

Now we consider the structure of the closure of homogeneous geodesics in Riemannian manifolds.
Let $(M,g)$ be a Riemannian manifold, and let $\gamma: \mathbb{R} \rightarrow M$ be a
geodesic parameterized with arc length, $\gamma(0)=x \in M$.
Suppose that $\gamma$ is homogeneous, i.~e. there is $U$ in the Lie algebra
$\mathfrak{g}$ of the Lie group $G=\Isom(M,g)$, such that
$\gamma(t)=\psi(t)(x)$, where $\psi(t)=\exp(Ut)$, $t\in \mathbb{R}$.
It is known that all orbits of any closed subgroup of $\Isom(M,g)$ on $M$ are
closed (see e.~g. Proposition 1 in \cite{Yau1977}). Therefore, for $\psi(\mathbb{R})$
closed in $G$, the geodesic $\gamma(\mathbb{R})$ is closed as the set in $M$. In case 2)
of Proposition \ref{prop.closopsub}, the geodesic $\gamma=\gamma(t)$ is periodic, but
it is possible also in case 3). For instance,  one can find in \cite{BerNik6, BerNik8}
several examples of Killing vector fields of constant length
that have both compact and non-compact integral curves
(such curves are homogeneous geodesics).
\medskip

Now, we assume that $\psi(\mathbb{R})$ is not closed in $G$.
The following result had been proved in \cite[Theorem 3.2]{Grove1974} for any complete
Riemannian manifold $(M,g)$  with compact isometry group $\Isom(M,g)$.
We prove a more general version using a similar approach.

\begin{theorem}\label{the.closhomgeod}
Let $(M,g)$ be a complete Riemannian manifold and let $\gamma: \mathbb{R} \rightarrow M$
be a homogeneous geodesic, i.~e. it is an orbit of some 1-parameter isometry group
$\psi(t)=\exp(Ut)$, $t\in \mathbb{R}$, for some $U \in \mathfrak{g}$, the Lie algebra of
the Lie group $G=\Isom(M,g)$.
Assume that $\gamma(\mathbb{R})$ is non-closed subset in $M$. Then
$\gamma(\mathbb{R})$ lies in a submanifold of $(M,g)$  diffeomorphic to a $l$-dimensional
torus $T^l$ with $l\geq 2$ and any
orbit of the group $\psi(t)$, $t\in \mathbb{R}$, through a point of $T^l$ is a geodesic
lying dense in $T^l$. Furthermore, the sectional curvature of any $2$-plane, tangent to
$T^l$ at a point $x\in T^l$ and containing $\gamma'(0)$, is nonnegative.
\end{theorem}

\begin{proof} Put $x=\gamma(0) \in M$.
Consider the closure $K$ of $\psi(\mathbb{R})$ in~$G$ and the closure $N$ of
$\gamma(\mathbb{R})$ in~$M$.
By Proposition \ref{prop.closopsub}, $K$ is a torus $T^k$ of dimension $k\geq 2$.

Note that $N$ is invariant under the action of every $\psi(t)$, $t\in \mathbb{R}$. Indeed,
if $x_0\in N$, then there are $t_n\in \mathbb{R}$ such that $\gamma(t_n)=\psi(t_n)(x)\to x_0$
as $n \to \infty$. Hence, $\psi(t)(\gamma(t_n))=\psi(t+t_n)(x)\to \psi(t)(x_0)$ as
$n \to \infty$, i.~e. $\psi(t)(x_0) \in N$.
Now, it is easy to see that $N$ is invariant even under the action of $K$. Moreover,
this action is transitive. Indeed, for every two points $a,b\in N$, there are a sequence
$t_n$ such that $a_n:=\psi(t_n)(x)\to a$ as $n\to \infty$ and
a sequence $s_n$ such that $b_n:=\psi(s_n)(x)\to b$ as $n\to \infty$. It is clear that
a sequence $\varphi_n = \psi(s_n-t_n)\in \psi(\mathbb{R})\subset K$
is such that $\varphi_n(a_n)=b_n$ for all $n$.
Passing, if necessary, to a subsequence, we can assume that $\varphi_n \to \varphi$ as
$n\to \infty$ for some $\varphi \in K \subset \Isom(M,g)$.
Hence, $b=\lim\limits_{n\to \infty} b_n =\lim\limits_{n\to \infty}\varphi_n(a_n) = \varphi (a)$,
$K$ acts transitively on $N$, and $N$ is the orbit of $K$ through the point $x$.
Hence, $N$ is a homogeneous space of a torus $K=T^k$, therefore, $N$ is a torus itself,
i.~e. $N=T^l$ with $l\geq 2$ (since $\gamma$ is not closed).

Note that $x$ is a critical point for the function
$y\in M \mapsto g_y(\widetilde{U},\widetilde{U})$, where $\widetilde{U}$ is a
Killing field, corresponding
to $U\in \mathfrak{g}$. It is easy to see that the value of this function is constant on
$N$. Hence, any point $y$ of $N$ is a critical point for the same function and integral
curve of $\widetilde{U}$ is a homogeneous geodesic. Since the distance between the points
$\psi(t_1)(x)$ and $\psi(t_2)(y)$ is equal to the distance between
the points $\psi(t_1+t)(x)$ and $\psi(t_2+t)(y)$ for every $t, t_1, t_2 \in \mathbb{R}$,
then the geodesic $\psi(t)(y)$, $t\in \mathbb{R}$,
is dense in $N$ for all $y \in N$.

Finally, let us prove the assertion on the sectional curvature.
Due to the previous assertion, we may (without loss of generality) consider only
points of the geodesic $\gamma$.
Suppose that $p=\gamma(s)$ for some $s\in \mathbb{R}$.
Let $Y_p$ be a unit tangent vector to $N=T^l$ at $p$ orthogonal to $\widetilde{U}_p$.
We define  the vector field $Y$ along $\gamma$ by setting $Y(t) = d(\psi(t))_p (Y_p)$.
By the construction of $Y$ and the previous discussion,
$Y$ is obtained with an 1-parameter geodesic variation of $\gamma$, i.~e. $Y$ is a
Jacobi field, therefore,
\begin{equation}\label{jac.equ}
\nabla^2 Y + R( Y, \gamma ')\gamma ' = 0.
\end{equation}
Taking inner product on both sides of  (\ref{jac.equ}) with $Y$ we obtain that the sectional
curvature of the $2$-plane spanned on $Y$ and $\gamma'$ (the latter is parallel to
$\widetilde{U}$ on $\gamma$) satisfies the equality
$K_{sec} = - g(\nabla^2 Y,Y)/g(\gamma', \gamma')$. Since
$g(Y, Y)$ is constant along $\gamma$, we have
\begin{eqnarray*}
2g(\nabla Y, Y)= \frac{d}{dt}\, g(Y, Y)=0, \\
0=\frac{d}{dt}\, g(\nabla Y, Y)= g(\nabla^2 Y,Y) + g(\nabla Y,\nabla Y),
\end{eqnarray*}
hence, $K_{sec} = g(\nabla Y,\nabla Y)/g(\gamma', \gamma') \geq 0$.
\end{proof}
\medskip

\begin{remark}
\label{nontrisotr}
Let $K_x$ be the isotropy subgroup of $K$ at the point $x\in N\subset M$. Then
$N=K/K_x$. It is interesting to study explicit examples with non-discrete $K_x$.
Note that, for any $a\in K_x$ (and, moreover, for any $a$ from the isotropy subgroup
of $\Isom(M,g)$ at the point $x$), the orbit of the group
$\xi(t)=\exp(t(\Ad(a)(U)))$, $t\in \mathbb{R}$, through $x$ is a geodesic.
\end{remark}

We have the following obvious corollary.

\begin{corollary}
\label{nonboundeded}
If a homogeneous geodesic $\gamma$ is not bounded in $(M,g)$ then $\gamma(R)$ is a closed subset in $M$.
\end{corollary}

Theorem \ref{the.closhomgeod} leads to the following natural questions (we use the
above notation).

\begin{question}\label{ques1}
Is it true that $N=T^l$ is totally geodesic in $(M,g)$?
\end{question}

\begin{question}\label{ques2}
Is it true that for any $W\in \operatorname{Lie}(K=T^k)$, the orbit of the group
$\exp(tW)$, $t\in \mathbb{R}$, trough every point of $N=T^l$ is a geodesic?
\end{question}

The above two questions are especially interesting for the case of homogeneous
Riemannian manifolds.
It is well known that these questions have positive answer for symmetric spaces $M$. In this
case, the closure of a given geodesic is either the geodesic itself or a totally geodesic torus
$N=T^l$, where $1 \leq l\leq \rk M$ and $\rk M$ is the rank of the symmetric space $M$.
In the next section we will show that generally both these questions have negative answers
even for the case of Lie groups with left-invariant Riemannian metrics.
\smallskip

There is another proof of the last assertion in Theorem \ref{the.closhomgeod}, which gives an
additional information connected with Questions~\ref{ques1} and \ref{ques2}. It is based on the
famous Gauss formula and equation. We shall use corresponding results from
\cite{KobNom1969}, with a little different notation.

Let $N$ be any smooth submanifold of a smooth Riemannian manifold $(M,g)$ with the induced metric
tensor $g'$ and Levi-Civita connection $\nabla'$. If $X,Y$ are vector fields on $N$ then
\textit{the Gauss formula} is
$$\nabla_XY=\nabla'_XY + \alpha(X,Y),$$
where $\alpha$ is \textit{the second fundamental form} of $N$ (in $M$) and $\alpha(X,Y)$ is
orthogonal to $N$.

By Proposition 4.5 in \cite{KobNom1969} we have
\begin{proposition}
\label{sect}
Let $X$ and $Y$ be a pair of orthonormal vectors in $N_x$, where $x\in N$. For 2-plane
$X\wedge Y$, spanned on $X$ and $Y$, we have
$$K_M(X\wedge Y)=K_N(X\wedge Y)+ g(\alpha(X,Y),\alpha(X,Y))-g(\alpha(X,X),\alpha(Y,Y)),$$
where $K_M$ (respectively, $K_N$) denotes the sectional curvature in $M$ (respectively, $N$).
\end{proposition}

Under conditions of Theorem \ref{the.closhomgeod}, we can take $X=\gamma'(0)$ and
any unit vector $Y$ at $x$, tangent to $N=T^l$ and orthogonal to $X$ (or even
corresponding parallel vector fields $X$, $Y$ on $N=T^l$). Then $\alpha(X,X)=0$ and
we immediately get the last statement of Theorem~\ref{the.closhomgeod} and the following
corollary.

\begin{corollary}
\label{nonclosed}
If $(M,g)$ has negative sectional curvature then $\gamma(R)$ is a closed subset in~$M$.
\end{corollary}

In general case, $N$ is a totally geodesic submanifold in $(M,g)$ if and only if $\alpha \equiv 0$
on~$N$. Therefore, using in our case the above parallel vector fields $X$, $Y$ on $N=T^l$ for
$l\geq 2$, we get the following corollary.

\begin{corollary}
\label{possect}
If $(M,g)$ has positive sectional curvature and $l\geq 2$ then $N=T^l$ is not totally geodesic
submanifold in $(M,g)$.
\end{corollary}

There are Riemannian manifolds of positive sectional curvature that have homogeneous
geodesics with the closure $N=T^l$ for $l\geq 2$. For instance, we can consider the
sphere $S^3=U(2)/U(1)$ supplied with $U(2)$-invariant Riemannian metrics (the Berger
sphere), that are sufficiently close to the metric of constant curvature in order to
have positive sectional curvature, see \cite[pp.~587--589]{Z2}.
Such metrics are naturally reductive, hence geodesic orbit i.~e. all their
geodesics are homogeneous.
It is easy to choose a geodesic that has the closure of dimension $2$. Indeed,
there are only countable set of periodic geodesics through a given point. It is
known that any self-intersecting geodesic in a homogeneous Riemannian manifold is
periodic. Then for all other geodesics the closure $N=T^l$ is such that $l=2$ (since
the torus $T^3$ could not act on $S^3$ effectively),
see details in \cite[Example 1]{Zi76}. Similar examples could be constructed for the sphere
$S^{2n-1}=U(n)/U(n-1)$ with any $n \geq 2$. By Corollary \ref{possect}, they provide
counterexamples to Question~\ref{ques1}, hence, Question~\ref{ques2}.
Note that the Berger spheres are weakly symmetric \cite{Zi96, Yak}. This shows that
the behaviour of geodesics in weakly symmetric spaces are more
complicated than in symmetric spaces.
\medskip

Interesting results on homogeneous geodesics in Riemannian manifolds of negative
sectional curvature were obtained in \cite[pp.~19--22]{BiONei69}, where such geodesics
were called Killing geodesics.

Interesting results on the behaviour of geodesics in homogeneous Riemannian spaces
are obtained also in \cite{Rodionov1981, Rodionov1984, Rodionov1990}.

\section{Examples of homogeneous geodesics}
\label{ehg}

Here we consider some examples of homogeneous geodesics on Riemannian manifolds. We
restrict ourself to
Lie groups with left-invariant Riemannian metrics. It is known that any homogeneous Riemannian
space $(G/H, g)$ admits at least one homogeneous geodesic~\cite{Ko-Szen}.
In the partial case of Lie groups with left-invariant Riemannian metrics, this result
was obtained earlier in~\cite{Kaj}.

Special examples of $3$-dimensional non-unimodular Lie group admit exactly one homogeneous
geodesic \cite{Marin2002}. On the other hand, in a three-dimensional unimodular Lie group $G$
endowed with a left-invariant metric $g$, there always exist three mutually orthogonal
homogeneous geodesics through each point. Moreover, for generic metrics,
there are no other homogeneous geodesics \cite{Marin2002}.

Let $G$ be a connected Lie group supplied with a left-invariant Riemannian metric
$g$, that is generated with some inner product $(\cdot,\cdot)$ on
$\mathfrak{g}=\operatorname{Lie}(G)$. We call $X \in \mathfrak{g}$ a {\it geodesic vector}
if $\exp(sX)$, $s\in \mathbb{R}$, is a geodesic in $(G,g)$
(i.~e. it is a homogeneous geodesic).
By Lemma \ref{GO-criterion}, we see that $X \in \mathfrak{g}$ is a geodesic vector if and only if
$([X,Y],X) =0$ for all $Y\in \mathfrak{g}$.

If $G$ is compact and semisimple, then the minus Killing form
$\langle \cdot,\cdot \rangle$ of $\mathfrak{g}$ is positive definite and
$\langle[X,Y],Z \rangle+\langle Y,[X,Z] \rangle=0$ for all $X,Y,Z \in \mathfrak{g}$. Hence,
all $X \in \mathfrak{g}$ are geodesic vectors
for the inner product $\langle \cdot,\cdot \rangle$.
For an arbitrary inner product $(\cdot,\cdot)$ on
$\mathfrak{g}$, there is a basis $E_1, E_2, \dots, E_n$, $n=\dim(G)$, in
$\mathfrak{g}$, simultaneously orthonormal for $\langle \cdot,\cdot \rangle$ and
orthogonal for $(\cdot,\cdot)$. Consider
numbers $\mu_i \in \mathbb{R}$ such that
$(E_i,E_i)=\mu_i \langle E_i,E_i \rangle$. Then for any $i$ and any $Y\in \mathfrak{g}$ we get
$$
([E_i,Y],E_i)=\mu_i \langle [E_i,Y],E_i \rangle=-\mu_i \langle Y_i,[E_i,E_i] \rangle=0,
$$
i.~e. $E_i$ is a homogeneous vector (see more general results for homogeneous spaces
in~\cite{Ko-Szen}).
\smallskip

Let us show that a small deformation of a given inner product could seriously change
the set of homogeneous vectors.

\begin{example}\label{exam1}
Suppose that a compact semisimple Lie algebra $\mathfrak{g}$ is supplied with the minus Killing
form $\langle \cdot,\cdot \rangle$. For any nontrivial $U\in \mathfrak{g}$, we have a
$\langle \cdot,\cdot \rangle$-orthogonal decomposition
$\mathfrak{g}=C_{\mathfrak{g}}(U) \oplus [U,\mathfrak{g}]$, where $C_{\mathfrak{g}}(U)$
is the centralizer of $U$ in $\mathfrak{g}$. Indeed, the operator $\ad(U)$ is skew
symmetric with respect to $\langle \cdot,\cdot \rangle$, hence $X\in C_{\mathfrak{g}}(U)$
implies $\langle X,[U,\mathfrak{g}] \rangle=0$ and the converse is also true.

Let us take any vector $V\in \mathfrak{g}$ that is not orthogonal both to $U$ and to
$[U,\mathfrak{g}]$ with respect to $\langle \cdot,\cdot \rangle$. Note that a generic
vector in $\mathfrak{g}$ has this property. Let us define the following inner product:
$(X,Y)=\langle X,Y \rangle+\alpha \cdot \langle X,V \rangle \cdot\langle Y,V \rangle$,
where $\alpha>0$.

If $\langle X,V \rangle=0$, then $(X,Y)=\langle X,Y \rangle$ for every $Y\in\mathfrak{g}$,
hence such $X$ is a geodesic vector for the inner product $(\cdot,\cdot)$. On the other hand,
there is $W\in \mathfrak{g}$ such that $\langle[U,W], V\rangle \neq 0$, hence,
$$
([W,U],W)= \langle [W,U],W \rangle+\alpha  \cdot\langle [W,U],V \rangle \cdot\langle U,V \rangle=\alpha
\cdot\langle [W,U],V \rangle \cdot\langle U,V \rangle \neq 0,
$$
$U$ is not a geodesic vector for $(\cdot,\cdot)$.

If $\rk\,\mathfrak{g} \geq 2  $, then we can choose the vector $X \in C_{\mathfrak{g}}(U)$
with the property $\langle X,V \rangle=0$. Moreover, let us fix a Cartan (i.~e. maximal
abelian) subalgebra $\mathfrak{t} \subset \mathfrak{g}$. We may consider the vector
$X\in \mathfrak{t}$ such that it is a regular vector in $\mathfrak{g}$
($C_{\mathfrak{g}}(X)=\mathfrak{t}$)
and the closure of the 1-parameter group $\exp(sX)$, $s\in \mathbb{R}$, coincides with
a maximal torus $T:=\exp(\mathfrak{t})$ in $G$.
If $U \in \mathfrak{t}$ with $\langle U,X \rangle=0$ and $V\in\mathfrak{g}$ such that
$\langle V,X \rangle=0$, $\langle V,U \rangle\neq 0$ and
$\langle V,[U,\mathfrak{g}] \rangle\neq 0$, then $X$ is ($U$ is not) a geodesic vector
for $(\cdot,\cdot)$ as above.

Hence, the closure of a homogeneous geodesic $\exp(sX)$, $s\in \mathbb{R}$, is the torus
$T$, but the orbit $\exp(sU)$, $s\in \mathbb{R}$, is not geodesic,
although $U\in\mathfrak{t} =\operatorname{Lie}(T)$. This gives the negative answer to
Question \ref{ques2}.

For Levi-Civita connection (elements of $\mathfrak{g}$ are considered as left-invariant
vector fields on $G$, see e.~g. \cite{Bes}) we have
\begin{eqnarray*}
2(\nabla _XU,Y)=2(\nabla _UX,Y)=([Y,U],X)+(U,[Y,X])\\
=\langle[Y,U],X\rangle +\langle U,[Y,X]\rangle+\alpha  \cdot\langle [Y,U],V \rangle \cdot\langle X,V \rangle
+\alpha  \cdot\langle U,V \rangle \cdot\langle [Y,X],V \rangle\\
=\alpha  \cdot\langle U,V \rangle \cdot\langle [Y,X],V \rangle \neq 0
\end{eqnarray*}
for some $Y\in [X,\mathfrak{g}]=[\mathfrak{t},\mathfrak{g}]$. Indeed,
$[U,\mathfrak{g}]\subset [X,\mathfrak{g}]$
(due to $C_{\mathfrak{g}}(X)=\mathfrak{t} \subset C_{\mathfrak{g}}(U)$), hence there is
$Y\in \mathfrak{g}$ such that $\langle[X,Y], V\rangle \neq 0$. This implies the required
result. We may assume also (without loss of generality) that this $Y$
is in $[X,\mathfrak{g}]=[\mathfrak{t},\mathfrak{g}]$
(since we have a $\langle \cdot,\cdot \rangle$-orthogonal decomposition
$\mathfrak{g}=C_{\mathfrak{g}}(X) \oplus [X,\mathfrak{g}]=\mathfrak{t}\oplus [X,\mathfrak{g}]$
and the $\mathfrak{t}$-component of $Y$ commutes with $X$).
Further, since the subspace $[X,\mathfrak{g}]=[\mathfrak{t},\mathfrak{g}]$ is
$\langle \cdot,\cdot \rangle$-orthogonal to $C_{\mathfrak{g}}(X)=\mathfrak{t}$, then the
torus $T$ is not totally geodesic with respect to the left-invariant
metric generated with the inner product $(\cdot, \cdot)$.
This gives the negative answer to Question~\ref{ques1}.
\end{example}

Now we are going to study geodesic vectors for some left-invariant Riemannian metrics
on the Lie group $G=SU(2)\times SU(2)$.

\begin{example}\label{exam2} Let us consider the basis $E_i$, $i=1,\dots,6$, in
$\mathfrak{g}=\mathfrak{su}(2)\oplus \mathfrak{su}(2)$ such that
the first three vectors are in the first copy of $\mathfrak{su}(2)$ whereas other vectors are in
the second copy of $\mathfrak{su}(2)$ and
\begin{eqnarray*}
{[E_1,E_2]=E_3}, \quad [E_2,E_3]=E_1, \quad [E_1,E_3]=-E_2,\\
{[E_4,E_5]=E_6}, \quad [E_5,E_6]=E_4, \quad [E_4,E_6]=-E_5.
\end{eqnarray*}
It is clear that this basis is orthonormal with respect to
the bi-invariant Riemannian metric $\langle \cdot , \cdot \rangle =-1/2 \cdot B(\cdot , \cdot)$,
where $B$ is the Killing form of $\mathfrak{g}$.
Let us consider a non-degenerate $(6\times 6)$-matrix $A=(a_{ij})$ and then inner
product $(\cdot, \cdot)$ such that the vectors $F_i=\sum_{j=1}^6 a_{ij} E_j$,
$i=1,\dots,6$, constitute a $(\cdot, \cdot)$-orthonormal basis.
It is easy to see that there is an one-to-one correspondence between the set of
all inner products on $\mathfrak{g}$ and the set of lower triangle matrices
$A=(a_{ij})$ with positive elements on the principal diagonal.

Let us consider the inner product $(\cdot, \cdot)_d$ that is generated with matrix
$$
A=
\left(
\begin{array}{cccccc}
1 & 0 & 0 & 0 & 0 & 0\\
0 & 1 & 0 & 0 & 0 & 0\\
0 & 0 & 1 & 0 & 0 & 0\\
1 & 0 & 0 & 1 & 0 & 0\\
1 & 0 & 0 & 0 & 1 & 0\\
d & 1 & 1 & 0 & 0 & d\\
\end{array}
\right),
$$
where $d>0$. Direct calculations (that could be performed with using any system of
computer algebra) show the the vector
$V=\sum_{i=1}^6 v_i E_i$ is a geodesic vector for the inner product $(\cdot, \cdot)_d$
if and only if one of the following conditions holds:

\begin{itemize}
\item $v_1 \in \mathbb{R}$, $v_2=v_3=0$, $v_4 \in \mathbb{R}$, $v_5 \in \mathbb{R}$,
$v_6=d\cdot v_1$;
\item $v_1 \in \mathbb{R}$, $v_2 \in \mathbb{R}$, $v_3 \in \mathbb{R}$,
$v_4 \in \mathbb{R}$, $v_5=2v_1-v_4$, $v_6=d\cdot v_1 +v_2+v_3$;
\item $v_1=d\cdot v_3$, $v_2=v_3$, $v_3 \in \mathbb{R}$, $v_4=d\cdot v_3$,
$v_5=d\cdot v_3$, $v_6 \in \mathbb{R}$.
\end{itemize}
Hence, the set of geodesic vectors for the metric $(\cdot, \cdot)_d$ is the union
of one 2-dimensional, one 3-dimensional, and one 4-dimensional linear subspaces in
$\mathfrak{g}=su(2)\oplus su(2)$.

The set of vectors $V=\sum_{i=1}^6 v_i E_i$ with $v_2=v_3=v_4=v_5=0$ determines a
Cartan subalgebra $\mathfrak{t}$ in $\mathfrak{g}$.
We see that $V \in \mathfrak{t}$ is a geodesic vector if either $v_1=0$, or $v_6=d\cdot v_1$.
Clear that in the latter case for any irrational $d$ we get the vector $V$ such that
the closure of the 1-parameter group $\exp(sV)$, $s\in \mathbb{R}$, coincides
with a maximal torus $T:=\exp(\mathfrak{t})$ in $G$.
Therefore, we again get the negative answer to Question \ref{ques2}.
\smallskip

If $D$ is the discriminant of the characteristic polynomial of the Ricci operator
$\operatorname{Ric}_d$ of the metric $(\cdot, \cdot)_d$, then
$$
D=\frac{3^{18}\cdot 17^2}{2^8 \cdot d^{44}} \bigl(1+o(d)\bigr)\quad \mbox{as} \quad d \to 0.
$$
Therefore, for sufficiently small positive $d$, $\operatorname{Ric}_d$
has 6 distinct eigenvalues. This implies that the full connected isometry group
of the metric $(\cdot, \cdot)_d$ is $G=SU(2)\times SU(2)$.
Indeed, if the dimension of full isometry group is $>6$, then the isotropy subgroup
is non-discrete and have (due to the effectiveness)
not only one-dimensional (hence, trivial) irreducible subrepresentations in
the isotropy representation, hence $\operatorname{Ric}_d$
should have some coincided eigenvalues.
Therefore, every operator $\ad(X)$, $X\in \mathfrak{g}=\mathfrak{su}(2)\oplus \mathfrak{su}(2)$,
is not skew symmetric with respect to $(\cdot, \cdot)_d$ for sufficiently small $d>0$.
\end{example}

Let us recall the following natural question.
\begin{question}\label{ques3}
Whether a given metric Lie algebra $(\mathfrak{g},Q)$ admits a basis that consists of
geodesic vectors (a geodesic basis)?
\end{question}

This question was studied in \cite{CHNT, CLNN, Ca-Ko-Ma, Ko-Szen}.
In \cite{Ko-Szen} it is
shown that semisimple Lie algebras possess an orthonormal basis comprised geodesic vectors,
for every inner product (it had been explained a little earlier in this paper). Results
for certain solvable algebras are given in \cite{Ca-Ko-Ma}.

It is easy to see that if $\mathfrak{g}$ possesses an orthonormal geodesic
basis (with respect to some inner product), then $\mathfrak{g}$ is unimodular \cite{CLNN}.
The authors of \cite{CLNN} proved that
every unimodular Lie algebra, of dimension at most $4$, equipped with an inner product,
possesses an orthonormal basis comprised geodesic vectors, whereas
there is an example of a solvable unimodular Lie algebra of dimension $5$ that has
no orthonormal geodesic basis, for any inner product.

The authors of \cite{CHNT} were interested in giving conditions for the Lie algebra
$\mathfrak{g}$ to admit a basis (not necessary orthonormal)
which is a geodesic basis with respect to some inner product.  The main results of
\cite{CHNT} show that the
following Lie algebras admit an inner product having a geodesic basis:
\begin{itemize}
\item unimodular solvable Lie algebras with abelian nilradical;
\item some Lie algebras with abelian derived algebra;
\item Lie algebras having a codimension one ideal of a particular kind;
\item unimodular Lie algebras of dimension $5$.
\end{itemize}
The authors of \cite{CHNT} also obtained
some negative results. For instance, they found the list of nonunimodular Lie
algebras of dimension $4$ admitting no geodesic basis.

\section{One special quadratic mapping}
\label{sqm}

We have discussed that the set of geodesic vectors on a given Lie algebra depends on
the chosen inner product $(\cdot,\cdot)$.
Let us consider this problem in a more general context.

Let $\mathfrak{g}$ be a Lie algebra, then every inner product $(\cdot,\cdot)$ on
$\mathfrak{g}$ determines a special quadratic mapping
\begin{equation}\label{spec.mapp}
\xi=\xi_{(\cdot,\cdot)}:\mathfrak{g} \mapsto \mathfrak{g}
\end{equation}
as follows: For any $X\in \mathfrak{g}$ we put $\xi(X):=V$, where $V$ is a unique vector in
$\mathfrak{g}$ with the equality $([X,Y],X)=(V,Y)$ for all $Y\in \mathfrak{g}$.

This mapping is well known, see e.~g. Proposition 7.28 in \cite{Bes} (where
$\xi(X):=-U(X,X)$ in the notation of this proposition)
for its generalization for homogeneous Riemannian spaces.

The set of zeros of the mapping $\xi_{(\cdot,\cdot)}$ is exactly the set of geodesic
vectors in $\mathfrak{g}$ with respect to the inner product $(\cdot,\cdot)$.
In particular, this set always contains the center of the Lie algebra $\mathfrak{g}$.
For any bi-invariant inner product $(\cdot,\cdot)$, the map
$\xi_{(\cdot,\cdot)}$ is obviously trivial. On the other hand,
this map could have unexpected properties for some special inner products $(\cdot,\cdot)$.

\begin{example}\label{exam3} Let us consider a basis $E_i$, $i=1,2,3$, in
$\mathfrak{g}=\mathfrak{su}(2)$ such that
$$
[E_1,E_2]=E_3, \quad [E_2,E_3]=E_1, \quad [E_1,E_3]=-E_2.
$$
Fix some positive numbers $a,b,c$ and consider the inner product $(\cdot,\cdot)$
on $\mathfrak{g}=\mathfrak{su}(2)$ that has an orthonormal basis
$F_1=aE_1$, $F_2=bE_2$, $F_3=cE_3$. Direct calculations give us an explicit form
of the mapping $\xi_{(\cdot,\cdot)}$
(we use coordinates of all vectors with respect to the original basis $E_1,E_2,E_3$):
$$
\xi_{(\cdot,\cdot)} \bigl( x_1,x_2,x_3 \bigr)=\left( \frac{a(b-c)}{bc}\, x_2x_3,  \frac{b(c-a)}{ac}\, x_1x_3, \frac{c(a-b)}{ab}\, x_1x_2 \right).
$$
It is easy to see that for $a\neq b\neq c\neq a$, any geodesic vector should be a
multiple of one of the vectors $E_1,E_2,E_3$.
On the other hand, for $a=b\neq c$, geodesic vectors are exactly the vectors either
with $x_3=0$ or with $x_1=x_2=0$.
For $a=b=c$ we have a bi-invariant inner product.
Obviously, $x_1x_2>0$ and $x_1x_3>0$ imply $x_2x_3>0$, hence, $\xi_{(\cdot,\cdot)}$
is not surjective. This correlates with Theorem 8 in \cite{ArZhuk2016}, stating (in
particular) that any surjective quadratic mapping $q:\mathbb{R}^3 \rightarrow \mathbb{R}^3$
has no non-trivial zero.
\end{example}

It could be an interesting problem to study general properties of the
quadratic mapping~(\ref{spec.mapp}) for general Lie algebras and general inner products.
Below we consider some results in this direction.
It should be recalled that the mapping $\xi_{(\cdot,\cdot)}$ always has at least one
non-trivial zero according to~\cite{Kaj}.
Nevertheless, this property does not imply directly the non-surjectivity of
$\xi_{(\cdot,\cdot)}$ for $\dim \mathfrak {g} \geq 6$, see
Example~3 and the corresponding discussion in \cite{ArZhuk2016}.
\smallskip

Recall that a Lie algebra $\mathfrak{g}$ is unimodular if $\trace \ad(Y)=0$ for all
$Y\in \mathfrak{g}$ (here, as usual, the operator
$\ad(Y):\mathfrak{g} \rightarrow \mathfrak{g}$ is defined with $\ad(Y)(Z)=[Y,Z]$). All
compact and semisimple Lie algebras are unimodular.
Any Lie algebra $\mathfrak{g}$ contains {\it the unimodular kernel}, the maximal
unimodular ideal, that could be described as follows:
$$
\mathfrak{u}:=\{ Y\in \mathfrak{g}\,|\,\trace \ad(Y)=0\}.
$$
For a non-unimodular Lie algebra $\mathfrak{g}$, the ideal $\mathfrak{u}$ has
codimension $1$ in $\mathfrak{g}$.
The following result is also well known (see e.~g. Lemma 7.32 in \cite{Bes}).

\begin{proposition}\label{qudmapgen.1}
Let $\mathfrak{g}$ be a Lie algebra supplied with an inner product
$(\cdot,\cdot)$, $\dim \, \mathfrak{g}=n$. Then the quadratic map
$\xi_{(\cdot,\cdot)}$ (see (\ref{spec.mapp})) has the following properties:

1) For a given $Y\in \mathfrak{g}$, the operator $\ad(Y)$ is $(\cdot,\cdot)$-skew
symmetric if and only if $(\xi(X),Y)=0$ for all $X\in \mathfrak{g}$.

2) Let us define $\Delta=\Delta_{(\cdot,\cdot)}\in \mathfrak{g}$ by the equation
$(\Delta,Y)=\trace \ad(Y)$, $Y \in \mathfrak{u}$.
Then $\Delta$ is $(\cdot,\cdot)$-orthogonal to the unimodular kernel $\mathfrak{u}$ of
$\mathfrak{g}$. In particular, $\Delta=0$ if~$\mathfrak{g}$ is unimodular.

3) $\Delta=-\sum_{i=1}^n \xi(E_i)$ for any $(\cdot,\cdot)$-orthonormal basis $E_i$, $i=1,\dots, n$,
in $\mathfrak{g}$.
\end{proposition}

\begin{proof}
Let us prove the first assertion.
Recall that the operator $\ad(Y)$ is $(\cdot,\cdot)$-skew symmetric if and only if
$([X,Y],X)=0$ for all $X\in \mathfrak{g}$.
By the definition of $\xi_{(\cdot,\cdot)}$  we get
$(\xi(X),Y)=([X,Y],X)$
for every $X\in \mathfrak{g}$, that proves the required assertion.

The second assertion follows directly from the definition of the unimodular kernel $\mathfrak{u}$.

Let us prove the third assertion. Fix any $Y \in \mathfrak{g}$. By the definition
of $\xi_{(\cdot,\cdot)}$ we have $(\xi(E_i),Y)=([E_i,Y],E_i)$. Therefore,
$$
\left(\sum_{i=1}^n \xi(E_i), Y \right)=\sum_{i=1}^n (\xi(E_i),Y)=-\sum_{i=1}^n ([Y,E_i],E_i)=-\trace \ad(Y)=-(\Delta,Y),
$$
that proves the required result.
\end{proof}
\smallskip

We hope that the further study of the mapping (\ref{spec.mapp}) will allow to understand more
deeply the set of geodesic vectors for general metric Lie algebras.

\bigskip

\bigskip

{\bf Acknowledgements.}
The first author was supported by the Ministry of Education and Science of the Russian Federation
(Grant 1.308.2017/4.6).

\vspace{10mm}

\bibliographystyle{amsunsrt}

\begin{thebibliography}{[99]}



\bibitem{AV}
Akhiezer, D.N., Vinberg, \`{E}.B.:
Weakly symmetric spaces and spherical varieties.
Transform. Groups. {\bf 4}, 3--24 (1999).


\bibitem{ArZhuk2016}
Arutyunov, A.V., Zhukovskiy, S.E.
Properties of surjective real quadratic maps.
Sb. Math. {\bf 207(9)}, 1187--1214 (2016).



\bibitem{Arv2017}
Arvanitoyeorgos, A.:
Homogeneous manifolds whose geodesics are orbits. Recent results and some open problems.
Irish Math. Soc. Bulletin. {\bf 79}, 5--29 (2017).


\bibitem{Bangert2016}
Bangert, V.:
Non-closed isometry-invariant geodesics.
Arch. Math. {\bf 106(6)}, 573--580 (2016).

\bibitem{BerNik6}
Berestovskii, V.N., Nikonorov, Yu.G.:
Killing vector fields of constant length on locally symmetric Riemannian manifolds.
Transform. Groups. {\bf  13(1)}, 25--45 (2008).

\bibitem{BerNik7}
Berestovskii, V.N., Nikonorov, Yu.G.:
Killing vector fields of constant length on Riemannian manifolds.
Siberian Math. J. {\bf 49(3)}, 395--407 (2008).


\bibitem{BerNik8}
Berestovskii, V.N., Nikonorov, Yu.G.:
Regular and quasiregular isometric flows on Riemannian manifolds.
Siberian Adv. Math. {\bf 18(3)}, 153--162 (2008).


\bibitem{BV1996}
Berndt J., Vanhecke L.:
Geometry of weakly symmetric spaces.
J. Math. Soc. Japan. {\bf 48(4)}, 745--760 (1996).


\bibitem{BKV}
Berndt, J., Kowalski, O., Vanhecke, L.:
Geodesics in weakly symmetric spaces.
Ann. Global Anal. Geom. {\bf 15(2)}, 153--156 (1997).


\bibitem{Bes}
Besse, A.L.:
Einstein Manifolds. Springer-Verlag, Berlin-Heidelberg-New York (1987).


\bibitem{BiONei69}
Bishop, R. L., O'Neill, B.:
Manifolds of negative curvature. Trans. Amer. Math. Soc. {\bf 145}, 1--49 (1969).


\bibitem{CHNT}
Cairns, G., Hini\'{c}-Gali\'{c}, A., Nikolayevsky, Y., Tsartsaflis, I.:
Geodesic bases for Lie algebras.
Linear Multilinear Algebra, {\bf 63}, 1176--1194 (2015).



\bibitem{CLNN}
Cairns, G., Le, N.T.T., Nielsen, A., Nikolayevsky, Y.:
On the existence of orthonormal geodesic bases for Lie algebras.
Note Mat. {\bf 33(23)}, 11--18 (2013).

\bibitem{Ca-Ko-Ma}
Calvaruso, G., Kowalski, O.,  Marinosci, R.A.:
Homogeneous geodesics in solvable Lie groups.
Acta Math. Hungar. {\bf 101(4)}, 313--322 (2003).





\bibitem{Gleason1949}
Gleason, A.M.:
On the structure of locally compact groups.
Proc. Nat. Acad. Sci. U. S. A. {\bf 35}, 384--386 (1949).


\bibitem{Goto1971}
Goto, M.:
Orbits of one-parameter groups. III. (Lie group case.).
J. Math. Soc. Japan. {\bf 23}, 95--102 (1971).


\bibitem{Grove1974}
Grove K.:
Isometry-invariant geodesics.
Topology {\bf 13}, 281--292 (1974).


\bibitem{HilNeeb}
Hilgert, J., Neeb, K.-H.:
Structure and geometry of Lie groups.
Springer Monographs in Mathematics. Springer, New York (2012).


\bibitem{Iwasawa1949}
Iwasawa, K.:
On some types of topological groups.
Ann. Math. (2) {\bf 50}, 507--558 (1949).




\bibitem{Kaj}
Ka\u{\i}zer, V.V.:
Conjugate points of left invariant metrics on Lie groups.
Sov. Math. {\bf 34(11)}, 32--44 (1990).

\bibitem{KobNom1969}
Kobayashi, Sh., Nomizu, K.:
Foundations of differential geometry. V. II. Interscience Publishers. New York-- London-- Sydney,
1969.



\bibitem{Ko-Szen}
Kowalski, O., Szenthe, J.:
On the existence of homogeneous geodesics in homogeneous Riemannian manifolds.
Geom. Dedicata. {\bf 81(1-3)}, 209--214 (2000); correction:
Ibid. {\bf 84(1-3)}, 331--332 (2001).



\bibitem{KV}
Kowalski, O., Vanhecke, L.:
Riemannian manifolds with homogeneous geodesics.
Boll. Unione Mat. Ital. Ser. B. {\bf 5(1)}, 189--246 (1991).


\bibitem{Marin2002}
Marinosci, R.A.:
Homogeneous geodesics in a three-dimensional Lie group. Comment. Math. Univ. Carolin. {\bf 43(2)}, 261--270 (2002).






\bibitem{MontZip1955}
Montgomery, D., Zippin, L.:
Topological transformation groups. Interscience Publishers, New York--London (1955).



\bibitem{Nikonorov2017}
Nikonorov, Yu.G.:
On the structure of geodesic orbit Riemannian spaces.
Ann. Glob. Anal. Geom. {\bf  52(3)}, 289--311 (2017).


\bibitem{Rodionov1981}
Rodionov, E.D.:
Homogeneous Riemannian Z-manifolds.
Sib. Math. J. {\bf 22(2)}, 315--320 (1981).


\bibitem{Rodionov1984}
Rodionov, E.D.:
Homogeneous Riemannian manifolds of rank one.
Sib. Math. J. {\bf 25(4)}, 642--644 (1984).


\bibitem{Rodionov1990}
Rodionov, E.D.:
Homogeneous Riemannian almost P-manifolds.
Sib. Math. J. {\bf 31(5)}, 789--794 (1990).


\bibitem{S}
Selberg, A.:
Harmonic analysis and discontinuous groups in weakly symmetric Riemannian spaces, with applications to Dirichlet series.
J. Indian Math. Soc. {\bf 20}, 47--87 (1956).



\bibitem{W1}
Wolf, J.A.:
Harmonic analysis on commutative spaces. Mathematical Surveys and Monographs, 142.
American Mathematical Society, Providence, RI (2007).



\bibitem{Yak}
Yakimova, O.S.:
Weakly symmetric Riemannian manifolds with a reductive isometry group.
Sb. Math. {\bf 195(3-4)}, 599--614 (2004).


\bibitem{Yau1977}
Yau, S.T.:
Remarks on the group of isometries of a Riemannian manifold.
Topology. {\bf 16(3)}, 239--247 (1977).



\bibitem{Zi76}
Ziller, W.:
Closed geodesics and homogeneous spaces.
Math. Z. {\bf 152}, 67--88 (1976).


\bibitem{Z2}
Ziller, W.:
The Jacobi equation on naturally reductive compact Riemannian homogeneous spaces.
Comment. Math. Helv. {\bf 52}, 573--590 (1977).


\bibitem{Zi96}
Ziller, W.:
Weakly symmetric spaces. In:
Topics in geometry. Progr. Nonlinear Differential Equations Appl., 20, pp.~355-368. Birkh{\"a}user, Boston, Boston, MA (1996).
\end{thebibliography}

\vspace{5mm}

\end{document}